\documentclass{article}
\usepackage[utf8]{inputenc}
\usepackage[margin=1 in]{geometry}
\usepackage{amsmath}
\usepackage{amsfonts}
\usepackage{amssymb}
\usepackage{amsthm}
\usepackage{mathrsfs}
\usepackage{tikz-cd}
\usepackage{enumitem}
\usepackage{lmodern}
\usepackage{hyperref}
\usepackage[vcentermath,enableskew]{youngtab}
\usepackage{bm}

\usepackage{graphbox}

\newtheorem{theorem}{Theorem}
\newtheorem{lemma}[theorem]{Lemma}

\newtheorem{conjecture}[theorem]{Conjecture}

\newtheorem*{prop*}{Proposition}

\theoremstyle{definition}
\newtheorem{definition}[theorem]{Definition}

\newtheorem{example}[theorem]{Example}

\newcommand{\tb}{\textbf}
\newcommand{\C}{\mathbb{C}}
\newcommand{\N}{\mathbb{N}}

\newcommand{\tn}{\textnormal}
\newcommand{\se}{\subseteq}

\newcommand{\til}{\widetilde}

\newcommand{\bs}{\backslash}

\newcommand{\ol}{\overline}
\newcommand{\lam}{\lambda}
\newcommand{\m}{\ol{\til{m}}}

\newcommand{\im}{\includegraphics[width=1cm,align=c]} 
\newcommand{\iml}{\includegraphics[width=1.4cm,align=c]} 
\newcommand{\imxl}{\includegraphics[width=1.8cm,align=c]} 
\newcommand{\ims}{\includegraphics[width=0.8cm,align=c]} 
\newcommand{\imxs}{\includegraphics[width=0.6cm,align=c]} 

\newcommand{\tcb}{\textcolor{blue}}

\usepackage[maxbibnames=10, style=alphabetic]{biblatex}
\title{Counting induced subgraphs with the Kromatic symmetric function}
\author{Laura Pierson \\ University of Waterloo \\ \href{mailto:lcpierson73@gmail.com}{lcpierson73@gmail.com}}

\begin{document}

\maketitle

\begin{abstract}
    The \emph{chromatic symmetric function} $X_G$ is a sum of monomials corresponding to proper vertex colorings of a graph $G$. Crew, Pechenik, and Spirkl (2023) recently introduced a $K$-theoretic analogue $\ol{X}_G$ called the \emph{Kromatic symmetric function}, where each vertex is instead assigned a nonempty set of colors such that adjacent vertices have nonoverlapping color sets. $X_G$ does not distinguish all graphs, but a longstanding open question is whether it distinguishes all trees. We conjecture that $\ol{X}_G$ does distinguish all graphs. As evidence towards this conjecture, we  show that $\ol{X}_G$ determines the number of copies in $G$ of certain induced subgraphs on 4 and 5 vertices as well as the number of induced subgraphs isomorphic to each graph consisting of a star plus some number of isolated vertices.
\end{abstract}

\section{Introduction}

The \emph{\tb{\tcb{chromatic symmetric function}}} was introduced by Stanley (1995) in \cite{stanley1995symmetric} as a symmetric function generalization of the chromatic polynomial. He remarked that he did not know of any two nonisomorphic trees with the same chromatic symmetric function, and much work has since been dedicated to studying which pairs of graphs are distinguished by $X_G.$ The chromatic symmetric function is known to distinguish all trees on up to 29 vertices (\cite{smith2015symmetric, heil2019algorithm}), various infinite families of trees including caterpillars and spiders (\cite{morin2005caterpillars, martin2008distinguishing, aliste2014proper, gerling2017distinguishing, aliste2023marked}), and several other infinite families of graphs including squids (\cite{martin2008distinguishing})  and trivially perfect graphs (\cite{tsujie2018chromatic}). Various generalizations of $X_G$ have also been shown to distinguish many or all trees, including a rooted version of $X_G$ (\cite{loehr2024rooted}), a group algebra version (\cite{foley2021transplanting}), a noncommutative version (\cite{gebhard2001chromatic}), and a quasisymmetric version (\cite{aval2023quasisymmetric}). Various properties of a tree are known to be computable from $X_G,$ including the subtree polynomial (\cite{martin2008distinguishing}), the path and degree sequence (\cite{martin2008distinguishing}), the number of vertices of degree at least 3 (\cite{crew2022note}), and counts of certain subtrees (\cite{lydon2016chromatic, salcido2023counting}). For general graphs (not necessarily trees), $X_G$ determines the girth and the number of vertices, edges, connected components, matchings, and triangles (\cite{orellana2014graphs}). However, it does not distinguish all graphs: \cite{orellana2014graphs} and \cite{aliste2021vertex} construct infinite families of pairs of nonisomorphic graphs with the same $X_G$, and \cite{crewgraphlist} lists 1000 pairs of small graphs with equal chromatic symmetric functions.

In \cite{kromatic2023}, Crew, Pechenik, and Spirkl (2023) introduced the following $K$-theoretic analogue of $X_G$ (see \cite{buch2005combinatorial} for background on combinatorial $K$-theory):

\begin{definition}[Crew, Pechenik, and Spirkl (2023), \cite{kromatic2023}]
    A \emph{\tb{\tcb{proper set coloring}}} of $G$ is a function $\kappa:V(G)\to 2^\N\bs\{\varnothing\}$ such that $\kappa(v)\cap \kappa(w) = \varnothing$ whenever $vw\in E(G)$, so each vertex receives a nonempty set of colors such that adjacent vertices have nonoverlapping color sets. For a vertex-weighted graph $G$ with weight function $\omega:V(G)\to\N,$ the \emph{\tb{\tcb{Kromatic symmetric function}}} is $$\ol{X}_{(G,\omega)} := \sum_\kappa \prod_{v\in V(G)}\left(\prod_{i\in \kappa(v)} x_i\right)^{\omega(v)}.$$ If no weight function is specified, we will assume all vertices have weight 1.
\end{definition}

A related function $Y_G$ was studied by Stanley (1998) in \cite{stanley1998graph}, with the difference that it includes terms where some vertices are not assigned any colors. Gasharov (1996, \cite{gasharov1996incomparability}) also introduced a similar function $\til{X}^{\boldsymbol{m}}_G$ that tracks colorings where vertex $i$ is assigned $m_i$ colors for a fixed sequence of nonnegative integers $\boldsymbol{m} = (m_1,m_2,\dots,)$. In \cite{marberg2023kromatic}, Marberg (2023) gives a construction of $\ol{X}_G$ using Hopf algebras and introduces several quasisymmetric analogues of $\ol{X}_G$.

Since $X_G$ can be obtained by taking the lowest degree terms of $\ol{X}_G$, $\ol{X}_G$ contains more information about $G$ than $X_G,$ and in fact, we conjecture that it contains enough information to distinguish all graphs:

\begin{conjecture}\label{conj:distinguishing}
    There do not exist nonisomorphic graphs $G$ and $H$ with $\ol{X}_G = \ol{X}_H$.
\end{conjecture}

As evidence for Conjecture \ref{conj:distinguishing}, we show that $\ol{X}_G$ can be used to count certain induced subgraphs of $G$. As an application of our results, we give an alternative proof that $\ol{X}_G$ distinguishes the graphs in each of three examples from \cite{kromatic2023} of pairs of graphs with the same $X_G$ but different $\ol{X}_G$ (see Example \ref{ex:distinguishing} below).

We prove three main results towards Conjecture \ref{conj:distinguishing}. Our first result counts certain subgraphs of order 4:

\begin{theorem}\label{thm:subgraphs_4}
    The number of induced copies in $G$ of the following order 4 graphs can be computed from $\ol{X}_G$:
    $$\arraycolsep=10pt
    \begin{array}{cccccccc}
        \imxs{4v} & \im{e+2v} & \ims{P3+v} & \imxs{2e} & \ims{claw} & \ims{K4-e} & \ims{K4} \\
    \end{array}$$
    and the counts of the remaining 4 order 4 graphs satisfy a system of 3 linear equations determined by $\ol{X}_G$.
\end{theorem}

In particular, Theorem \ref{thm:subgraphs_4} implies that $\ol{X}_G$ can determine whether or not $G$ contains an induced claw \ims{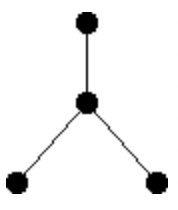}, which is of interest because $X_G$ cannot tell whether $G$ is claw-free and much is known about the structure of claw-free graphs (\cite{chudnovsky2005structure}). Claw-free graphs are of particular interest in the study of the chromatic symmetric function because another major open question about $X_G$ is the Stanley-Stembridge conjecture (\cite{stanley1995symmetric}, Conjecture 5.1), which says that $X_G$ is $e$-positive for $G$ a claw-free incomparability graph of a partially ordered set.

Our second result counts certain subgraphs of order 5:

\begin{theorem}\label{thm:subgraphs_5}
    The number of induced copies of the following order 5 graphs can also be computed from $\ol{X}_G$:
    $$\arraycolsep=7pt
    \begin{array}{ccccccccccc}
        \ims{5v} & \im{e+3v} & \im{P3+2v} & \im{2e+v} & \ims{claw+v} & \im{K14} & \ims{K4-e+P3} & \im{K3+e_5} & \im{cricket} & \iml{P5} & \im{P3+e}
    \end{array}$$
     and the counts of the remaining 23 order 5 graphs satisfy a system of 14 linear equations determined by $\ol{X}_G.$ Also, for $T$ a tree, the number of copies of all induced subgraphs of order 4 or 5 can be determined from $\ol{X}_T.$ 
\end{theorem}

\begin{example}\label{ex:distinguishing}
    The authors of \cite{kromatic2023} give examples of three pairs of graphs with the same $X_G$ but different $\ol{X}_G$. We can use Theorems \ref{thm:subgraphs_4} and \ref{thm:subgraphs_5} to distinguish all three pairs:
    \begin{itemize}
        \item For \im{example_1_1} and \im{example_1_2}, the second graph has an induced \ims{K4-e} while the first does not.
        \item For \im{example_2_1} and \im{example_2_2}, the first graph has an induced \im{cricket} while the second does not.
        \item For \imxl{example_3_1} and \imxl{example_3_2}, the second graph has an induced \ims{K4-e+P3} while the first does not.
    \end{itemize}
\end{example}

Our third result concerns an infinite family of subgraphs that can be counted using $\ol{X}_G$:

\begin{theorem}\label{thm:stars}
    For each pair $h,k \ge 0,$ the number of induced copies in $G$ of the disjoint union of an $h$-vertex star and $k$ isolated vertices can be recovered from $\ol{X}_G.$
\end{theorem}

In \cite{martin2008distinguishing}, Martin, Morin, and Wagner (2008) proved the stronger fact that the the degree sequence of a tree can be recovered from the ordinary symmetric function $X_G.$ Our proof of Theorem \ref{thm:stars} is shorter and uses different methods, which suggests that proving that $\ol{X}_G$ distinguishes all trees may be easier than proving that $X_G$ does.

The remainder of this paper is organized as follows: in Section \ref{sec:background}, we give some background on graph theory, symmetric functions, and the chromatic symmetric function. In Section \ref{sec:lemma}, we prove a key lemma that will be used in several of our proofs. In Sections \ref{sec:subgraphs_4}, \ref{sec:subgraphs_5}, and \ref{sec:stars}, we give the proofs of our three main results.

\section{Background}\label{sec:background}

Let $\N = \{1,2,3,\dots\}$ denote the set of positive integers, and $2^\N$ its power set.

A \emph{\tb{\tcb{graph}}} $G$ is a finite set $V(G)$ of \emph{\tb{\tcb{vertices}}} together with a set $E(G)$ of \emph{\tb{\tcb{edges}}} such that each edge is an unordered pair $vw$ of vertices $v,w\in V(G).$ Two vertices $v$ and $w$ are \emph{\tb{\tcb{adjacent}}} or \emph{\tb{\tcb{neighbors}}} if $vw\in E(G).$ An \emph{\tb{\tcb{isolated vertex}}} is a vertex with no neighbors. The \emph{\tb{\tcb{order}}} of $G$ is $|V(G)|$. A \emph{\tb{\tcb{tree}}} is a connected graph with no cycles. The \emph{\tb{\tcb{complete graph}}} $K_n$ is the graph on $n$ vertices with every two vertices adjacent. A \emph{\tb{\tcb{star graph}}} is a graph where one vertex is adjacent to all others but no other pairs of vertices are adjacent. Two graphs $G$ and $H$ are \emph{\tb{\tcb{isomorphic}}} if there is a bijection $\phi:V(G)\to V(H)$ such that $$E(H) = \{\phi(v)\phi(w):vw\in E(G)\}.$$ A \emph{\tb{\tcb{subgraph}}} of $G$ is a graph $H$ with $V(H) \se V(G)$ and $E(H) \se E(G).$ An \emph{\tb{\tcb{induced subgraph}}} is a subgraph $H$ with $V(H) \se V(G)$ and $$E(H) = \{vw:v,w\in V(H),vw\in E(G)\}.$$ An induced subgraph is a \emph{\tb{\tcb{clique}}} if every two vertices are adjacent and a \emph{\tb{\tcb{stable set}}} if no two vertices are adjacent. The \emph{\tb{\tcb{complement graph}}} of $G$ is the graph $\ol{G}$ given by $$V(\ol{G}) = V(G), \ \ \ E(\ol{G}) = \{vw:v,w\in V(G),vw\not\in E(G)\}.$$ 

A \emph{\tb{\tcb{partition}}} $\lam = (\lam_1,\dots,\lam_{\ell(\lam)})$ is a nondecreasing sequence $\lam_1\ge \dots \ge \lam_{\ell(\lam)}$ of positive integers, and the numbers $\lam_1,\dots,\lam_{\ell(\lam)}$ are its \emph{\tb{\tcb{parts}}}. We write $\lam = \ell^{i_\ell}(\ell-1)^{i_{\ell-1}}\dots 3^{i_3}2^{i_3}1^{i_1}$ to denote the partition with $i_j$ parts of size $j$ for each $j = 1,2,\dots,\ell.$ The \emph{\tb{\tcb{length}}} $\ell(\lam)$ is the number of parts in $\lam.$

A \emph{\tb{\tcb{symmetric function}}} is a power series $f(x_1,x_2,\dots) \in \C[[x_1,x_2,x_3,\dots]]$ such $$f(x_1,x_2,x_3,\dots) = f(x_{\sigma(1)},x_{\sigma(2)},x_{\sigma(3)},\dots)$$ for any permutation $\sigma$ of $\N.$ For a partition $\lam,$ the \emph{\tb{\tcb{monomial symmetric function}}} $m_\lam$ is $$m_\lam(x_1,x_2,\dots) := \sum_{\substack{i_1,i_2,\dots,i_{\ell(\lam)}\in \N \\ \tn{ pairwise distinct}}}x_{i_1}^{\lam_1}x_{i_2}^{\lam_2}\dots x_{i_{\ell(\lam)}}^{\lam_{\ell(\lam)}}.$$ For each $d,$ the monomial symmetric functions of degree $d$ form a basis for the vector space of homogeneous symmetric functions of degree $d.$

A \emph{\tb{\tcb{proper coloring}}} $\kappa$ of a graph $G$ is a function $\kappa:V(G) \to \N$ such that if $vw\in E(G),$ then $\kappa(v)\ne \kappa(w).$ The \emph{\tb{\tcb{chromatic symmetric function}}} (introduced by Stanley (1995) \cite{stanley1995symmetric}) is $$X_G := \sum_\kappa \prod_{v\in V(G)}x_{\kappa(v)},$$ where $\kappa$ ranges over all proper colorings of $G$.

\section{Key lemma}\label{sec:lemma}

Our main tool will be the expansion formula given in \cite{kromatic2023} for $\ol{X}_G$ in the $\m$-basis, which they define as the following generalization of the $m$-basis above:

\begin{definition}[Crew, Pechenik, and Spirkl (2023), \cite{kromatic2023}]
    The \emph{\tb{\tcb{$K$-theoretic augmented monomial function}}} associated to a partition $\lam$ is $$\m_\lam := \ol{X}_{K_\lam},$$ where $K_\lam := (K_{\ell(\lam)},\omega)$ is the vertex-weighted complete graph on $\ell(\lam)$ vertices with weights $\omega(i) := \lam_i.$
\end{definition}

They show that $\ol{X}_G$ can be written as a linear combination of the $\m_\lam$'s with positive integer coefficients:

\begin{theorem}[Crew, Pechenik, and Spirkl (2023), \cite{kromatic2023}]\label{thm:m_expansion}
    For each $\lam,$ the coefficient $[\m_\lam]$ of $\m_\lam$ in the $\m$-expansion of $\ol{X}_G$ counts the number of ways to cover $V(G)$ with $\ell(\lam)$ distinct (but possibly overlapping) stable sets whose sizes are the parts of $\lam.$
\end{theorem}

Our proofs will make use of the following lemma, which follows from Theorem \ref{thm:m_expansion}:

\begin{lemma}\label{lem:subgraphs}
    For any $i_1,\dots,i_\ell,$ the number $\#(i_1,i_2,\dots,i_\ell)$ of ways to cover exactly $i_1$ of the vertices of $\ol{G}$ using $i_2$ edges of $\ol{G}$, $i_3$ $K_3$'s in $\ol{G}$, $\dots,$ and $i_\ell$ $K_\ell$'s in $\ol{G}$ can be determined from $\ol{X}_G.$ In particular, the number $\#(i,j)$ of (not necessarily induced) subgraphs of the complement graph $\ol{G}$ with $i$ vertices, $j$ edges, and no isolated vertices can be determined from $\ol{X}_G.$ 
\end{lemma}

\begin{proof}
    By Theorem \ref{thm:m_expansion}, $[\m_\lam]$ counts the number of ways to cover $V(G)$ with stable sets whose sizes are the parts of $\lam.$ Note first that the number of vertices $n := |V(G)|$ can be determined from $\ol{X}_G$ because $[\m_{1^k}] = 0$ for all $k\ne n$ while $[\m_{1^n}] = 1.$ 
    
    To compute $\#(i_1,i_2,\dots,i_\ell),$ note that by Theorem \ref{thm:m_expansion}, $[\m_{\ell^{i_\ell}(\ell-1)^{i_{\ell-1}}\dots 3^{i_3}2^{i_2}1^{n-i_1}}]$ counts the number of ways to cover $V(G)$ using $n-i_1$ singletons and $i_j$ stable sets of size $j$ for each $j = i_2,\dots,i_\ell.$ Translating this to $\ol{G},$ $[\m_{\ell^{i_\ell}(\ell-1)^{i_{\ell-1}}\dots 3^{i_3}2^{i_2}1^{n-i_1}}]$ counts the number of ways to cover $\ol{G}$ with $n-i_1$ singletons together with $i_j$ \emph{cliques} $K_j$ for each $j.$ Now we split this count into cases based on how many of the vertices covered by one of the $n-i_1$ singletons are also covered by one of the cliques $K_j.$ The number of ways to choose one of the covers counted by $[\m_{\ell^{i_\ell}(\ell-1)^{i_{\ell-1}}\dots 3^{i_3}2^{i_2}1^{n-i_1}}]$ that includes at least $k$ vertices that are each covered by at least one stable set besides a singleton is $\#(k,i_2\dots,i_\ell)\binom k{i_1}.$ To see this, note first that by the definition of our notation, there are $\#(k,i_2,\dots,i_\ell)$ ways to cover $k$ vertices using $\ell_j$ cliques $K_j$ for each $j$. Then, $n-k$ of the singletons are needed to cover the remaining $n-k$ uncovered vertices, so $(n-i_1) - (n-k) = k - i_1$ singletons will cover one of the $k$ already covered vertices. So, there are $\binom k{k-i_1} = \binom k{i_1}$ ways to choose which $k-i_1$ of those $k$ vertices are covered by a singleton in addition to another stable set. Putting this together and then rearranging to isolate the value $\#(i_1,i_2,\dots,i_\ell)$ that we want to solve for, we get $$\#(i_1,i_2,\dots,i_\ell) = [\m_{\ell^{i_\ell}(\ell-1)^{i_{\ell-1}}\dots 3^{i_3}2^{i_2}1^{n-i_1}}] - \sum_{k=i_1+1}^n \#(k,i_2\dots,i_\ell)\binom k{i_1}.$$ Thus, we can recursively compute $\#(i_1,i_2,\dots,i_\ell)$ by fixing $i_2,\dots,i_\ell$ and letting $i_1$ range from $n$ down to 1, so the values $\#(i_1,i_2,\dots,i_\ell)$ are all computable from $\ol{X}_G$ by induction.
    \end{proof}

    \section{Proof of Theorem \ref{thm:subgraphs_4}}\label{sec:subgraphs_4}

    Finding induced subgraphs of $G$ isomorphic to $H$ is equivalent to finding induced subgraphs of $\ol{G}$ isomorphic to $\ol{H}$, so we will instead focus on induced subgraphs of $\ol{G},$ as Lemma \ref{lem:subgraphs} makes them easier to think about. We will use a picture of $H$ to represent the number of induced subgraphs of $\ol{G}$ isomorphic to $H,$ and unless stated otherwise, all references to subgraphs and edges will be for $\ol{G}$ rather than for $G.$ The pictures of the graphs are taken from \cite{smallgraphlist}.

    From \cite{martin2008distinguishing}, the number of vertices, edges, triangles, and induced copies of \im{P3} in $G$ can be computed from $X_G,$ and thus also from $\ol{X}_G$ since the lowest degree terms of $\ol{X}_G$ determine $X_G.$ For $|V(H)|=4$, there are 11 total choices of $H.$ We can directly compute 5 of the values:
    \begin{align}
        \ims{K4} &= [\m_{41^{n-4}}] \label{eqn:K4} \\ 
        \ims{K4-e} &= \#(4,5) - 6\cdot\ims{K4} \label{eqn:K4-e} \\
        \ims{K3+e} &= (\ims{e}-3)\cdot \imxs{K3} - [\m_{321^{n-5}}] - 12\cdot\ims{K4} - 4\cdot\ims{K4-e}  \label{eqn:K3+e} \\
        \ims{C4} &= \#(4,4) - 15\cdot \ims{K4} - 5\cdot \ims{K4-e} - \ims{K3+e} \label{eqn:C4} \\
        \im{K3+v} &= (n-3)\cdot \imxs{K3} - 4\cdot \ims{K4} - 2\cdot \ims{K4-e} - \ims{K3+e}. \label{eqn:K3+v}
    \end{align}
    (\ref{eqn:K4}) is immediate, and the reasoning for the other equations is as follows:
    \begin{itemize}
        \item (\ref{eqn:K4-e}): We count subgraphs of $\ol{G}$ with 5 edges and 4 vertices, of which there are $\#(4,5)$ total. For each \ims{K4}, there are 6 such subgraphs since we can choose any of the 6 edges to omit, and for each \ims{K4-e}, there is one. 
        
        \item (\ref{eqn:K3+e}): We count ways to choose a triangle plus an edge not contained in the triangle, of which there are  $(\ims{e}-3)\cdot \imxs{K3}$ total. Then $[\m_{321^{n-5}}]$ counts the ways where the triangle and the edge do not overlap, for each \ims{K4} there are 12 ways (4 options for the triangle, then 3 for the edge), for each \ims{K4-e} there are 4 ways (2 options for the triangle, then 2 for the edge), and for each \ims{K3+e} there is one way.
        
        \item (\ref{eqn:C4}): We count ways to cover 4 vertices with 4 edges. There are $\#(4,4)$ ways total. For each \ims{K4}, there are $\binom 64 = 15$ ways since we choose which 4 edges to include, for each \ims{K4-e}, there are $\binom 54 = 5$ ways, and for each \ims{C4} or \ims{K3+e}, there is just one way since we need all 4 edges. 
        
        \item (\ref{eqn:K3+v}): We count ways to choose a triangle plus a vertex not contained in the triangle. There are $(n-3)\cdot\imxs{K3}$ ways total. For each \ims{K4}, there is one way since \ims{K4} has 4 triangles, for each \ims{K4-e}, there are 2 ways since it has 2 triangles, and for each \ims{K3+e} or \im{K3+v}, there is one way.
    \end{itemize}
    For the remaining 6 cases where $|V(H)|=4$, we set up a system of 5 equations in 6 variables:
    $$\begin{pmatrix}
        1 & 1 & 1 & 1 & 1 & 1 \\
        3 & 3 & 2 & 2 & 1 & 0 \\
        1 & 1 & 0 & 0 & 0 & 0 \\
        1 & 0 & 0 & 1 & 0 & 0 \\
        2 & 3 & 1 & 0 & 0 & 0 \\
    \end{pmatrix}
    \begin{pmatrix}
        \im{P4} \\ \\
        \imxs{claw} \\ \\
        \ims{P3+v} \\ \\
        \imxs{2e} \\ \\
        \im{e+2v} \\ \\
        \imxs{4v}
    \end{pmatrix}
    = \begin{pmatrix}
        b_1 \\ b_2 \\ b_3 \\ b_4 \\ b_5 
    \end{pmatrix},$$ where
    \begin{align*}
        b_1 &= \binom n4 - \ims{K4} - \ims{K4-e} - \ims{K3+e} - \ims{C4} - \im{K3+v} \\
        b_2 &= \binom{n-2}2\cdot \ims{e} - 6\cdot \ims{K4} - 5\cdot\ims{K4-e} - 4\cdot \ims{K3+e} - 4\cdot \ims{C4} - 3\cdot \im{K3+v} \\
        b_3 &= \#(4,3) - 16\cdot\ims{K4} - 8\cdot\ims{K4-e} - 3\cdot \ims{K3+e} - 4\cdot \ims{C4} \\
        b_4 &= [\m_{221^{n-4}}] - 3\cdot\ims{K4} - 2\cdot\ims{K4-e} - \ims{K3+e} - 2\cdot\ims{C4} \\
        b_5 &= (n-3)\cdot\im{P3} - 2\cdot\ims{K4-e} - 2\cdot \ims{K3+e} - 4\cdot \ims{C4}
    \end{align*}

    The explanations for these equations are as follows:
    \begin{itemize}
        \item \tb{Row 1:} We count all ways to choose four vertices, which on the one hand is $\binom n4,$ and on the other is the sum over all $H$ with $|V(H)|=4$ of the number of induced subgraphs isomorphic to $H$.
        \item \tb{Row 2:} We count ways to choose an edge plus two other vertices, which on the one hand is $\binom{n-2}2$ times the number of edges, and on the other is the sum over all $H$ of the number of induced copies of $H$ times the number of edges in $H$.
        \item \tb{Row 3:} We count ways to cover 4 vertices with 3 edges. The total number of ways is \#(4,3), and for each $H$, the number of ways is $\binom{|E(H)|}3$ minus the number of triangles in $H$, since to cover all 4 vertices of $V(H)$ with 3 edges from $E(H),$ we can choose any 3 edges that do not form a triangle. 
        \item \tb{Row 4:} We count ways to choose two nonoverlapping edges. The total number of ways is $[\m_{221^{n-4}}],$ and for each $H$, the number of ways equals the number of pairs of nonoverlapping edges in $H$.
        \item \tb{Row 5:} We count ways to choose an induced \im{P3} plus an extra vertex. The total number of ways is $(n-3)\cdot\im{P3}$, and the number of ways for each $H$ is the number of induced copies of \im{P3} in $H$.
    \end{itemize}
    
    The reduced row echelon form of the above matrix is
    $$\begin{pmatrix}
        1&0&0&1&0&0\\
        0&1&0&-1&0&0\\
        0&0&1&1&0&0\\
        0&0&0&0&1&0\\
        0&0&0&0&0&1\\
    \end{pmatrix}.$$
    Since the last two rows each have a single 1 and the rest 0's, the number of copies in $\ol{G}$ of the corresponding graphs can be computed from $\ol{X}_G$, so the number of copies in $\ol{G}$ of the following 7 graphs can be determined:
    $$\arraycolsep=10pt
    \begin{array}{ccccccc}
        \ims{K4} & \ims{K4-e} & \ims{K3+e} & \ims{C4} & \im{K3+v} & \im{e+2v} & \imxs{4v} \\
    \end{array}$$
    Taking complements, the number of copies in $G$ of the following 7 graphs can be determined:
    $$\arraycolsep=10pt
    \begin{array}{cccccccc}
        \imxs{4v} & \im{e+2v} & \ims{P3+v} & \imxs{2e} & \ims{claw} & \ims{K4-e} & \ims{K4} \\
    \end{array}$$
    Finally, note that for $T$ a tree, this lets us determine the induced number of copies in $T$ of all order 4 forests except \im{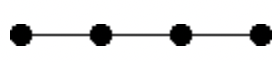}. But since $T$ a tree, the number of copies of \im{P4.png} can then also be computed since it is the only remaining possible 4 vertex subgraph of $T$. \qed

    \section{Proof of Theorem \ref{thm:subgraphs_5}}\label{sec:subgraphs_5}

    There are 34 graphs with $|V(H)|=5.$ We first directly compute 6 of the values:
    \begin{align}       
        \ims{K5} &= [\m_{51^{n-5}}] \label{eqn:K5} \\
        \ims{K5-e} &= \#(5,9) - 10\cdot \ims{K5} \label{eqn:K5-e} \\
        \im{K4+P3} &= \#(5,0,1,1) - 30\cdot \ims{K5} - 6\cdot\ims{K5-e} \label{eqn:K4+P3} \\
        \imxs{W4} &= \#(5,8) - 45\cdot \ims{K5} - 9\cdot\ims{K5-e} - \im{K4+P3} \label{eqn:W4} \\
        \im{K4+e} &= (\ims{e}-6)\cdot\ims{K4} - [\m_{421^{n-6}}] - 20\cdot\ims{K5} - 6\cdot\ims{K5-e} - 2\cdot\im{K4+P3} \label{eqn:K4+e} \\
        \im{K4+v} &= (n-4)\cdot\ims{K4} - 5\cdot\ims{K5} - 2\cdot\ims{K5-e} - \im{K4+P3} - \im{K4+e} \label{eqn:K4+v}
    \end{align}
    (\ref{eqn:K5}) is immediate, and the other equations are obtained as follows:
    \begin{itemize}
        \item (\ref{eqn:K5-e}): We count ways to cover 5 vertices with 9 edges. The total is $\#(5,9),$ for each \ims{K5}, there are 10 ways since any edge can be removed, and for each \ims{K5-e}, there is one way.
        \item (\ref{eqn:K4+P3}): We count ways to cover 5 vertices with a triangle and a $K_4.$ For each \ims{K5}, there are 5 ways to choose the $K_4$ and then $\binom42 = 6$ ways to choose the triangle, for $5\cdot6=30$ total, for each \ims{K5-e}, there are 2 ways to choose the $K_4$ and then 3 ways to choose the triangle, for $2\cdot3=6$ total, and for each \im{K4+P3}, there is one way.
        \item (\ref{eqn:W4}): We count ways to cover 5 vertices with 8 edges. The total is $\#(5,8).$ For each \ims{K5}, there are $\binom{10}2 = 45$ ways since we can choose any 2 edges to remove, for each \ims{K5-e}, there are 9 ways since we can choose any one edge to remove, and for each \imxs{W4} or \im{K4+P3}, there is one way.
        \item (\ref{eqn:K4+e}): We count ways to choose a $K_4$ plus an extra edge. There are $(\ims{e} - 6)\cdot\ims{K4}.$ ways total. There are $[\m_{421^{n-6}}]$ ways where the edge and triangle do not overlap, 20 for each \ims{K5}, 6 for each \ims{K5-e}, 2 for each \im{K4+P3}, and one for each \im{K4+e}.
        \item (\ref{eqn:K4+v}): We count ways to choose a $K_4$ and another vertex. There are $(n-4)\cdot\ims{K4}$ total, and for each 5 vertex graph $H$, the number of ways is the number of $K_4$ subgraphs of $H.$
    \end{itemize}
    For the remaining 28 values, we set up the system of 19 equations shown in Figure \ref{fig:matrix}, with the aid of Sage (\cite{sage}), where $\mathbf{b}$ is computable from $\ol{X}_G.$ The rows of the matrix are obtained as follows:
    \begin{figure}
        \centering
        \caption{System of linear equations to count induced subgraphs of order 5}
        \label{fig:matrix}
    $$\arraycolsep=3pt
    \begin{pmatrix}
5 & 2 & 1 & 1 & 1 & 0 & 0 & 0 & 0 & 0 & 0 & 0 & 0 & 0 & 0 & 0 & 0 & 0 & 0 & 0 & 0 & 0 & 0 & 0 & 0 & 0 & 0 & 0\\
0 & 3 & 2 & 0 & 0 & 4 & 2 & 1 & 3 & 1 & 1 & 1 & 0 & 0 & 0 & 0 & 0 & 0 & 0 & 0 & 0 & 0 & 0 & 0 & 0 & 0 & 0 & 0\\
0 & 0 & 0 & 0 & 0 & 0 & 0 & 0 & 2 & 0 & 1 & 0 & 0 & 0 & 2 & 0 & 0 & 0 & 0 & 0 & 2 & 1 & 0 & 0 & 0 & 0 & 1 & 0\\
0 & 0 & 0 & 0 & 0 & 0 & 0 & 0 & 0 & 0 & 1 & 2 & 0 & 0 & 0 & 2 & 2 & 0 & 0 & 0 & 0 & 1 & 4 & 0 & 2 & 2 & 1 & 2\\
0 & 0 & 0 & 0 & 0 & 0 & 0 & 0 & 0 & 0 & 0 & 0 & 1 & 1 & 0 & 0 & 0 & 3 & 0 & 0 & 0 & 0 & 0 & 0 & 1 & 0 & 0 & 2\\
0 & 0 & 0 & 0 & 0 & 0 & 0 & 0 & 0 & 0 & 0 & 0 & 0 & 0 & 1 & 0 & 1 & 0 & 3 & 0 & 0 & 0 & 0 & 0 & 0 & 2 & 1 & 1\\
0 & 0 & 0 & 0 & 0 & 1 & 1 & 2 & 0 & 2 & 0 & 0 & 0 & 2 & 0 & 1 & 0 & 0 & 0 & 3 & 3 & 3 & 1 & 5 & 2 & 1 & 2 & 0\\
0 & 0 & 0 & 1 & 4 & -1 & 0 & -2 & 0 & 1 & 0 & 2 & 0 & 1 & 0 & 0 & 1 & 2 & 2 & -1 & -3 & -1 & -1 & 0 & 0 & 0 & 0 & 0\\
0 & 0 & 2 & 4 & 4 & 0 & 2 & 2 & 0 & 2 & 2 & 2 & 4 & 2 & 2 & 2 & 2 & 2 & 2 & 2 & 0 & 0 & 0 & 0 & 0 & 0 & 0 & 0\\
0 & 0 & 0 & 0 & 1 & 0 & 0 & 0 & 0 & 1 & 0 & 3 & 0 & 4 & 0 & 3 & 8 & 12 & 20 & 1 & 1 & 4 & 11 & 5 & 12 & 23 & 9 & 26\\
0 & 0 & 0 & 0 & 0 & 0 & 0 & 0 & 0 & 0 & 0 & 1 & 0 & 1 & 0 & 1 & 5 & 6 & 18 & 0 & 0 & 1 & 6 & 1 & 6 & 19 & 5 & 20\\
0 & 0 & 0 & 0 & 0 & 0 & 0 & 0 & 0 & 0 & 0 & 0 & 0 & 0 & 0 & 0 & 1 & 1 & 7 & 0 & 0 & 0 & 1 & 0 & 1 & 7 & 1 & 7\\
0 & 0 & 0 & 0 & 0 & 0 & 0 & 0 & 0 & 0 & 0 & 0 & 0 & 0 & 0 & 0 & 0 & 0 & 1 & 0 & 0 & 0 & 0 & 0 & 0 & 1 & 0 & 1\\
0 & 0 & 0 & 0 & 0 & 0 & 0 & 0 & 0 & 0 & 0 & 0 & 0 & 0 & 0 & 0 & 0 & 0 & 0 & 0 & 1 & 1 & 2 & 0 & 1 & 2 & 1 & 2\\
0 & 0 & 0 & 0 & 0 & 0 & 0 & 0 & 0 & 0 & 0 & 1 & 0 & 0 & 0 & 1 & 4 & 0 & 12 & 0 & 0 & 1 & 6 & 0 & 3 & 14 & 4 & 10\\
0 & 0 & 0 & 0 & 0 & 0 & 0 & 0 & 0 & 0 & 0 & 0 & 0 & 0 & 0 & 0 & 2 & 0 & 12 & 0 & 0 & 0 & 2 & 0 & 1 & 12 & 2 & 8\\
0 & 0 & 0 & 0 & 0 & 0 & 0 & 0 & 0 & 0 & 0 & 0 & 0 & 0 & 0 & 0 & 0 & 0 & 3 & 0 & 0 & 0 & 0 & 0 & 0 & 3 & 0 & 2\\
0 & 0 & 0 & 0 & 0 & 0 & 0 & 0 & 0 & 0 & 0 & 0 & 0 & 0 & 0 & 0 & 0 & 0 & 0 & 0 & 0 & 0 & 1 & 0 & 0 & 1 & 0 & 0\\
0 & 0 & 0 & 0 & 0 & 0 & 0 & 0 & 0 & 0 & 0 & 0 & 0 & 0 & 0 & 0 & 1 & 0 & 6 & 0 & 0 & 0 & 0 & 0 & 0 & 5 & 1 & 2\\
    \end{pmatrix}
    \begin{pmatrix}
        \imxs{5v} \\ \\ \ims{e+3v} \\ \\ \ims{P3+2v} \\ \\ \ims{claw+v} \\ \ims{K14} \\ \ims{2e+v} \\ \\ \im{P4+v} \\ \\ \ims{P3+e} \\  \\ \imxs{K3+2v} \\ \imxs{chair} \\ \ims{K3+e+v} \\ \ims{cricket} \\ \ims{C4+v} \\ \imxs{flag} \\ \ims{K4-e+v} \\ \imxs{bull} \\ \ims{dart} \\ \ims{K23} \\ \imxs{K4-e+P3} \\ \im{P5} \\ \ims{K3+e_5} \\ \ims{K3+P3} \\ \ims{2K3} \\ \imxs{C5} \\ \imxs{house} \\ \ims{fan} \\ \im{kite} \\ \ims{C4+claw} 
    \end{pmatrix} = \mathbf{b}$$
    \end{figure}
    \begin{itemize}
        \item \tb{Rows 1-9:} We count induced copies of the order 4 subgraphs in the order 5 subgraphs.
        \item \tb{Rows 10-13:} We count connected subgraphs with each number of edges from 4 to 7.
        \item \tb{Rows 14-17:} We count ways to cover 5 vertices with a triangle plus 1, 2, 3, or 4 edges.
        \item \tb{Rows 18-19:}  We count ways to cover 5 vertices with 2 triangles, or 2 triangles plus an edge.
    \end{itemize}
    The reduced row echelon form of the matrix is:
    $$\arraycolsep=3pt
    \begin{pmatrix}
1 & 0 & 0 & 0 & 0 & 0 & 0 & 1 & 0 & 1 & 0 & 0 & 0 & 0 & 0 & 1 & 0 & 0 & 0 & 2 & 0 & 0 & 0 & 2 & 0 & 0 & -1 & 0\\
0 & 1 & 0 & 0 & 0 & 0 & 0 & -3 & 0 & -3 & 0 & 0 & 0 & 0 & 0 & -3 & 0 & 0 & 0 & -6 & 0 & 0 & 0 & -6 & 0 & 0 & 3 & 0\\
0 & 0 & 1 & 0 & 0 & 0 & -1 & 1 & 0 & 1 & 0 & 0 & 0 & 0 & 0 & 2 & 0 & 0 & 0 & 3 & 0 & 0 & 0 & 3 & 0 & 0 & -2 & 0\\
0 & 0 & 0 & 1 & 0 & 0 & 1 & 0 & 0 & -1 & 0 & 0 & 0 & 0 & 0 & -1 & 0 & 0 & 0 & -2 & 0 & 0 & 0 & -2 & 0 & 0 & 1 & 0\\
0 & 0 & 0 & 0 & 1 & 0 & 0 & 0 & 0 & 1 & 0 & 0 & 0 & 0 & 0 & 0 & 0 & 0 & 0 & 1 & 0 & 0 & 0 & 1 & 0 & 0 & 0 & 0\\
0 & 0 & 0 & 0 & 0 & 1 & 1 & 2 & 0 & 2 & 0 & 0 & 0 & 0 & 0 & 1 & 0 & 0 & 0 & 3 & 0 & 0 & 0 & 3 & 0 & 0 & -1 & 0\\
0 & 0 & 0 & 0 & 0 & 0 & 0 & 0 & 1 & 0 & 0 & 0 & 0 & 0 & 0 & 0 & 0 & 0 & 0 & 0 & 0 & 0 & 0 & 0 & 0 & 0 & 0 & 0\\
0 & 0 & 0 & 0 & 0 & 0 & 0 & 0 & 0 & 0 & 1 & 0 & 0 & 0 & 0 & 0 & 0 & 0 & 0 & 0 & 0 & -1 & 0 & 0 & 0 & 0 & -1 & 0\\
0 & 0 & 0 & 0 & 0 & 0 & 0 & 0 & 0 & 0 & 0 & 1 & 0 & 0 & 0 & 1 & 0 & 0 & 0 & 0 & 0 & 1 & 0 & 0 & 0 & 0 & 0 & 0\\
0 & 0 & 0 & 0 & 0 & 0 & 0 & 0 & 0 & 0 & 0 & 0 & 1 & 0 & 0 & 0 & 0 & 0 & 0 & 0 & 0 & 0 & 0 & -1 & 0 & 0 & 0 & 0\\
0 & 0 & 0 & 0 & 0 & 0 & 0 & 0 & 0 & 0 & 0 & 0 & 0 & 1 & 0 & 0 & 0 & 0 & 0 & 0 & 0 & 0 & 0 & 1 & 0 & 0 & 0 & 0\\
0 & 0 & 0 & 0 & 0 & 0 & 0 & 0 & 0 & 0 & 0 & 0 & 0 & 0 & 1 & 0 & 0 & 0 & 0 & 0 & 0 & 0 & 0 & 0 & 0 & 0 & 0 & 0\\
0 & 0 & 0 & 0 & 0 & 0 & 0 & 0 & 0 & 0 & 0 & 0 & 0 & 0 & 0 & 0 & 1 & 0 & 0 & 0 & 0 & 0 & 0 & 0 & 0 & -1 & 1 & 0\\
0 & 0 & 0 & 0 & 0 & 0 & 0 & 0 & 0 & 0 & 0 & 0 & 0 & 0 & 0 & 0 & 0 & 1 & 0 & 0 & 0 & 0 & 0 & 0 & 0 & 0 & 0 & 0\\
0 & 0 & 0 & 0 & 0 & 0 & 0 & 0 & 0 & 0 & 0 & 0 & 0 & 0 & 0 & 0 & 0 & 0 & 1 & 0 & 0 & 0 & 0 & 0 & 0 & 1 & 0 & 0\\
0 & 0 & 0 & 0 & 0 & 0 & 0 & 0 & 0 & 0 & 0 & 0 & 0 & 0 & 0 & 0 & 0 & 0 & 0 & 0 & 1 & 1 & 0 & 0 & 0 & 0 & 1 & 0\\
0 & 0 & 0 & 0 & 0 & 0 & 0 & 0 & 0 & 0 & 0 & 0 & 0 & 0 & 0 & 0 & 0 & 0 & 0 & 0 & 0 & 0 & 1 & 0 & 0 & 1 & 0 & 0\\
0 & 0 & 0 & 0 & 0 & 0 & 0 & 0 & 0 & 0 & 0 & 0 & 0 & 0 & 0 & 0 & 0 & 0 & 0 & 0 & 0 & 0 & 0 & 0 & 1 & 0 & 0 & 0\\
0 & 0 & 0 & 0 & 0 & 0 & 0 & 0 & 0 & 0 & 0 & 0 & 0 & 0 & 0 & 0 & 0 & 0 & 0 & 0 & 0 & 0 & 0 & 0 & 0 & 0 & 0 & 1\\
    \end{pmatrix}.$$
    There are 5 rows with a single 1 and the rest 0's, so combining those with our initial 6 graphs, we find that $\ol{X}_G$ determines the number of induced copies in $\ol{G}$ of the following graphs:
    $$\arraycolsep=7pt
    \begin{array}{ccccccccccc}
        \ims{K5} & \ims{K5-e} & \im{K4+P3} & \ims{W4} & \im{K4+e} & \im{K4+v} & \ims{K3+2v} & \ims{K23} & \im{K4-e+v} & \ims{house} & \im{C4+claw}
    \end{array}.$$
    Taking complements, $\ol{X}_G$ determines the counts in $G$ of the following 11 graphs:
    $$\arraycolsep=7pt
    \begin{array}{ccccccccccc}
        \ims{5v} & \im{e+3v} & \im{P3+2v} & \im{2e+v} & \ims{claw+v} & \im{K14} & \ims{K4-e+P3} & \im{K3+e_5} & \im{cricket} & \iml{P5} & \im{P3+e}
    \end{array}.$$
    
    For $T$ a tree, this gives the counts of all order 5 forests except \im{P4+v} and \imxs{chair}. We can then compute the number of copies of \imxs{chair} because we know the number of copies of each induced subgraph on 4 vertices from the proof of Theorem \ref{thm:subgraphs_4}, and each \imxs{chair} has an induced \ims{claw} while each \im{P4+v} does not. Then all remaining sets of 5 vertices would form an induced \im{P4+v}, so that number can be determined as well. \qed

    \section{Proof of Theorem \ref{thm:stars}}\label{sec:stars}

        An induced copy of a $j$-vertex star in $G$ is equivalent to an induced copy in $\ol{G}$ of the order $j$ graph consisting of a $K_{j-1}$ together with an isolated vertex, which we will denote as $K_{j-1}\sqcup v.$ For $1\le i\le j-1,$ let $K_{j-1} \sqcup i\cdot e$ denote the order $j$ graph consisting of a $K_{j-1}$ plus an extra vertex $v$ connected to exactly $i$ of the vertices in the clique. Thus, the complement $\ol{K_{j-1}\sqcup i\cdot e}$ is an $(j-i)$-vertex star together with $i$ isolated vertices, since when we take the complement, the vertex $v$ not in the $K_{j-1}$ becomes the center of the star and the vertices of the $K_{j-1}$ adjacent to $v$ become the isolated vertices.
        
        For a graph $H$, write $\#(H)$ for the number of induced copies of $H$ in $\ol{G}.$ By Lemma \ref{lem:subgraphs}, the number $$\#(\underbrace{j,i,0\dots,0,1}_{j-1})$$ of ways to cover $j$ vertices in $\ol{G}$ with $i$ edges and one $K_{j-1}$ can be computed from $\ol{X}_G$ for each $i.$ Thus, our strategy will be to set up a system of linear equations by expressing these known values in terms of the numbers $\#(K_{j-1}\sqcup v)$ and $\#(K_{j-1}\sqcup i\cdot e)$ that we are interested in computing.
        
        Every set of $j$ vertices in $H$ that can be covered by $i$ edges and a $K_{j-1}$ is contained $K_{j-1}\sqcup k\cdot e$ for some $k,$ since those are the only induced subgraphs on $j$ vertices that contain a $K_{j-1}.$ Given such an induced subgraph of $H$ isomorphic to $K_{j-1}\sqcup k\cdot e,$ if we want to cover it with a $K_{j-1}$ and $i$ edges, we must first choose the $K_{j-1}$. Thus, let $a_{k}$ denote the number of induced copies of $K_{j-1}$ in $K_{j-1}\sqcup k\cdot e.$ If $k = j-1,$ then all vertices of the $K_{j-1}$ are connected to the extra vertex $v,$ so $K_{j-1}\sqcup (j-1)\cdot e = K_j.$ There are $j$ induced copies of $K_{j-1}$ in $K_j$ is $j,$ so $a_{j-1}=j.$ If $k = j-2,$ we get $a_{j-2}=2$ induced copies of $K_{j-1}$ in $K_{j-1}\sqcup (j-2)\cdot e,$ since the vertex not in the $K_{j-1}$ can be either of the two nonadjacent vertices. For $k < j-2,$ there is just one copy of $K_{j-1}$ in $K_{j-1}\sqcup k\cdot e,$ since the vertex with only $j-2$ neighbors must be the one left out. Thus, $$a_k = \begin{cases}
            j, &\tn{if }k=j-1; \\
            2, &\tn{if }k=j-2; \\
            1, &\tn{if }k < j-2.
        \end{cases}$$
        Now that we have the $K_{j-1}$, let $v$ be the remaining vertex not in the $K_{j-1}$. We need to choose $i$ of the $\binom{j-1}2+k$ total edges in the $K_{j-1}\sqcup k\cdot e$ such that at least one of them has $v$ as an endpoint (in order two cover $v.$ The number of ways to do this is $$\binom{\binom{j-1}2+k}{i} - \binom{\binom{j-1}2}{k},$$ since the first term represents the number of ways to choose any $i$ edges of $K_{j-1}\sqcup k\cdot e$ and the second represents the number of ways to choose all $k$ edges from the $K_{j-1}$ (which we do not want since then $v$ is not covered). 
        
        Since $i$ can take any value from 1 to $j-1$, we have a system of $j-1$ equations $$\sum_{k=1}^{j-1} a_k\left(\binom{\binom{j-1}2+k}{i} - \binom{\binom{j-1}2}{i}\right)\cdot\#(K_{j-1}\sqcup k\cdot e) = \#(\underbrace{j,i,0\dots,0,1}_{j-1})$$ for $i = 1,2,\dots,j-1.$ This gives $j-1$ linearly independent equations in the $j-1$ variables $\#(K_{j-1}\sqcup e),\#(K_{j-1}\sqcup 2\cdot e),\dots,\#(K_{j-1}\sqcup(j-1)\cdot e),$ so we can solve for all the variables to compute the values $\#(K_{j-1}\sqcup j\cdot e)$ for $k = 1,\dots,j-1$ given $\ol{X}_G.$
        
        Then to compute $\#(K_{j-1}\sqcup v),$ we count the total number of ways to choose a $K_{j-1}$ in $\ol{G}$ plus an extra vertex. The total number of $K_{j-1}$'s in $\ol{G}$ is just $[\m_{(j-1)1^{n-j+1}}],$ and the remaining vertex can be any of the other $n-j+1$ vertices in $\ol{G}$, so the total number of ways is $(n-j+1)\cdot[\m_{(j-1)1^{n-j+1}}].$ The $j$ chosen vertices will then form either an induced $K_{j-1}\sqcup v$ or an induced $K_{j-1}\sqcup k\cdot e$ in $\ol{G}$ for some $i,$ since those are the only induced subgraphs of order $j$ that contain a $K_{j-1}.$ For each induced $K_{j-1}\sqcup k\cdot e$ in $\ol{G}$, there are $a_k$ ways to choose the $K_{j-1},$ and the the extra vertex would have to be the remaining vertex. For each induced $K_{j-1}+v,$ there is one choice for the clique and the vertex. Thus, we get $$\#(K_{j-1}\sqcup v) = (n-j+1)\cdot[\m_{(j-1)1^{n-j+1}}] - \sum_{k=1}^{j-1}a_k\cdot \#(K_{j-1}\sqcup k\cdot e).$$ Since all values on the right side can be computed from $\ol{X}_G$, so can $\#(K_{j-1}\sqcup v).$ 
        
        Since we were counting induced subgraphs of $\ol{G}$, taking complements tells us that for each $j$ and $k,$ we can compute of copies in $G$ of $\ol{K_{j-1}\sqcup k\cdot e},$ which is an a star on $h = j-k$ vertices together with $k$ isolated vertices, as well as the number of copies in $G$ of $K_{j-1}\sqcup v,$ which is a star on $h=j$ vertices. This completes the proof of Theorem \ref{thm:stars}. \qed

\section*{Acknowledgements}

The author thanks Oliver Pechenik for suggesting the problem and for helpful discussions and comments. She was partially supported by the Natural Sciences and Engineering Research Council of Canada (NSERC) grant RGPIN-2022-03093.

\printbibliography

\end{document}